\documentclass[a4paper, 11pt]{article}
\usepackage[T1]{fontenc}
\usepackage[utf8]{inputenc}
\usepackage{amsfonts}
\usepackage{textcomp}

\usepackage{enumerate} 
\usepackage{hyperref} 

\usepackage{amssymb}
\usepackage[leqno]{amsmath}
\usepackage{amsthm} 

\theoremstyle{definition}

\numberwithin{defn}{section}
\newtheorem{rmk}{Remark}
\numberwithin{rmk}{section}
\theoremstyle{plain}
\newtheorem{thm}{Theorem}
\numberwithin{thm}{section}
\newtheorem{prop}[thm]{Proposition}
\newtheorem{lemma}[thm]{Lemma}
\newtheorem{cor}[thm]{Corollary}

\usepackage[dvipsnames]{xcolor} 

\usepackage{tikz}
\usepackage{tikz-cd}

\usepackage{mathrsfs} 

\newcommand{\yo}{\text{\usefont{U}{min}{m}{n}\symbol{'110}}}
\DeclareFontFamily{U}{min}{}
\DeclareFontShape{U}{min}{m}{n}{<-> dmjhira}{}

\usepackage{xspace} 
\newcommand{\eg}{e.g.\@\xspace} 
\newcommand{\ie}{i.e.\@\xspace}

\mathchardef\hyphen="2D 

\renewcommand{\epsilon}{\varepsilon}
\renewcommand{\phi}{\varphi}

\newcommand{\dom}{\operatorname{dom}}
\newcommand{\cod}{\operatorname{cod}}

\newcommand{\Eq}{\operatorname{Eq}}

\newcommand{\AND}{\wedge}

\newcommand{\inv}{^{-1}}
\newcommand{\op}{^{\textup{op}}}

\newcommand{\ran}{\mathrm{ran}}

\newcommand{\colim}{\mathrm{colim}}

\newcommand{\sheafify}{\mathrm{a}}

\newcommand{\Set}{\mathbf{Set}}

\newcommand{\cbicat}{\mathcal{C}}
\newcommand{\dbicat}{\mathcal{D}}

\newcommand{\Etopos}{\mathscr{E}}
\newcommand{\Ftopos}{\mathscr{F}}

\newcommand{\powerset}{\mathscr{P}}

\usepackage{tikz}
\makeatletter
\newbox\xrat@below
\newbox\xrat@above
\newcommand{\xrightarrowtail}[2][]{%
	\setbox\xrat@below=\hbox{\ensuremath{\scriptstyle #1}}%
	\setbox\xrat@above=\hbox{\ensuremath{\scriptstyle #2}}%
	\pgfmathsetlengthmacro{\xrat@len}{max(\wd\xrat@below,\wd\xrat@above)+.6em}%
	\mathrel{\tikz [>->,baseline=-.75ex]
		\draw (0,0) -- node[below=-2pt] {\box\xrat@below}
		node[above=-2pt] {\box\xrat@above}
		(\xrat@len,0) ;}}
\makeatother

\date{August 22, 2019} 

\title{On the dependent product in toposes}
\author{Olivia Caramello \hspace{0.2cm} and \hspace{0.2cm} Riccardo Zanfa}

\begin{document}
	
	\maketitle{}

\begin{abstract}
	We give an explicit construction of the dependent product in an elementary topos, and a site-theoretic description for it in the case of a Grothendieck topos. 
\end{abstract}
	
\section*{Introduction}

Given a morphism $f:P\rightarrow Q$ in a topos $\Etopos$, the functor
$$\textstyle\prod_f:\Etopos/P\rightarrow \Etopos/Q$$
right adjoint to the pullback functor $f^*:\Etopos/Q\rightarrow \Etopos/P$ is called the \emph{dependent product}\footnote{For the name we refer to \cite{nlab:dependent_product}.} along $f$. In fact, among its numerous applications, this construction provides a sound interpretation of dependent products in type theory.

In the present note we provide a description for it different from those available in the literature, both in the setting of elementary toposes and in that of Grothendieck toposes.

Before presenting our construction, let us quickly revise the classical approaches to the problem. 

There is a standard strategy to prove the existence of the dependent product which recurs in most topos theory books:
\begin{enumerate}[(i)]
	\item one starts from the particular case $Q=1$, that is $f:P\rightarrow 1$, to build $\prod_{f}$;
	\item for a general $f$, the equivalence $(\Etopos/Q)/f\simeq \Etopos/P$ brings us back to item (i). 
\end{enumerate}
So the key part is the first, while the second only requires to translate the same process in a slice topos. 

Let us see how (i) is delivered in some standard references on the subject:
\begin{itemize}
	\item For $f:P\rightarrow 1$ and $h:H\rightarrow P$, $\prod_{f}(h)$ is the pullback of $h^P:H^P\rightarrow P^P$ along $1_\Etopos\rightarrow P^P$, the mate arrow of $1_P:P\rightarrow P$. 
	
	This approach can be found in \cite[Theorem I.9.4]{maclanemoerdijk}, \cite[Corollary A1.5.3]{elephant} and \cite[Theorem 11.7]{mclarty}. 
	
	\item An alternative method for proving the existence of $\prod_{f}$ for $f:P\rightarrow 1$ is based on the following fact: a logical functor between toposes (that is, a cartesian functor which preserves power objects, equivalently exponentials and the subobject classifier) has a left adjoint if and only if it has a right adjoint. So, in order to conclude that $(f:P\rightarrow 1)^*$ has a right adjoint, it suffices to show $(f:P\rightarrow 1)^*$ is logical and admits the forgetful functor $\Etopos/P\rightarrow \Etopos$ as its left adjoint. The second claim is obvious, while the first is established in \cite[Corollary 1.43]{johnstone.topos} by using partial arrow classifiers (see Definition 1.25 there), and in \cite[Theorem 5.8.4]{borceux3} and \cite[Chapter 5, Corollary 3.7]{barrwells} by using properties of the power-object functor $\powerset$ as a monadic functor.
	
	\item A description of $\prod_{f}$ in terms of the internal language can be found in \cite[Theorem 4.23]{bell}: for $h:H\rightarrow P$ and $f:P\rightarrow 1$ the dependent product $\prod_{f}(h)$ is defined syntactically as the object $\prod_{p\in P}h\inv(p)$.
\end{itemize}
For the sake of completeness we remark that in \cite[pag. 451]{goldblatt} the approach is slightly different, in that it exploits partial arrows in a way similar to \cite{johnstone.topos} but in a one-step explicit proof. Recall that a partial arrow from $A$ to $B$ is a span $(A\xleftarrow{g}D\xrightarrow{f} B)$, where $g$ is a monomorphism, and that every topos admits a classifier for partial arrows to an object $B$, that is, an arrow $\eta_{B}:B\to \tilde{B}$ such that for every partial arrow $(A\xleftarrow{g}D\xrightarrow{f}B)$ from $A$ to $B$ there is a unique arrow $\tilde{f}:A\to \tilde{B}$ such that the following square is a pullback:
\begin{equation*}
\begin{tikzcd}
D \ar[r, tail, "g"] \ar[d, "f"'] & A \ar[d, "\tilde{f}"]\\
B \ar[r, "\eta_{B}"] & \tilde{B}
\end{tikzcd}
\end{equation*}

The dependent product $\prod_f(h)$ is then constructed by taking the pullback of $\tilde{h}^P$ along the arrow $k':Q\rightarrow \tilde{P}^P$ whose mate $k:Q\times P\rightarrow \tilde{P}$ is the classifying arrow of the partial morphism  $(Q\times P\xleftarrow{\langle f,1\rangle}P\xrightarrow{1_P}P)$. 

These approaches, however, make it difficult to handle the dependent product concretely: as easy as the procedure in (i) can be, the cartesian closed structure in slice toposes is rather convoluted to describe, and this makes the translation in (ii) cumbersome. The same is true for alternative strategies that resort to the partial arrow classifier; furthermore, partial arrows are nowadays not very cared for in topos theory.

In the present note we provide an alternative way to describe the dependent product different from those listed above. We do not claim this to be \textquoteleft better' in any sense, but it has the perk of being a one-step construction directly relying on the two essential properties of an elementary topos: the existence of power objects and of finite limits. 

Section 1 treats the case where $\Etopos$ is any elementary topos. Our construction is based on the possibility, ensured by the ``boundedness'' of the operations involved in the construction, of modelling (the underlying object of) a dependent product as the domain of a subobject of a finite product of basic objects and powersets on them. This allows us to reduce the computation to that of the value of the $\forall$ functor at very simply defined subobjects, which is in turn shown to admit an elementary description in terms of power objects and finite limits. Our identification of the relevant subobjects stems from an analysis of $\prod_f$ in the topos $\Set$, and from the subsequent identification of a formula in the internal language of $\mathcal{E}$ defining the dependent product. Lastly, it is shown that the dependent product behaves very naturally with respect to subtoposes.

Section 2 analyses the case where $\Etopos$ is a Grothendieck topos: our description of the dependent product relies on representing the slice toposes ${\cal E}\slash P$ and ${\cal E}\slash Q$ as categories of sheaves on suitable sites such that the geometric morphism ${\cal E}\slash P \to {\cal E}\slash Q$ whose inverse image is $f^{\ast}$ is induced by a comorphisms between them. 

\subsection*{Notation}

We shall denote objects in slice toposes with square brackets to distinguish them from the underlying morphism: \ie $f\in \textup{Mor}(\Etopos)$ while $[f]\in \Etopos/Q$. 

We shall write, as is usual in the category-theoretic literature, $F \dashv G$ to mean that a functor $F$ is left adjoint to a functor $G$. 

The unique arrow from $P$ to the terminal $1_\Etopos$ will be denoted by $!_P$; when $Q=1_\Etopos$ we will write $P^*\dashv\prod_P$ instead of $(1_P)^*\dashv \prod_{!_P}$. The canonical geometric morphism $\Etopos \slash P \to \Etopos\slash Q$ induced by an arrow $f:P\to Q$, whose direct image is $\prod_f$ and whose inverse image is $f^*$, will be denoted by $l^{\cal E}_{f}$. If $Q=1_{\cal E}$ then we shall also write $l^{\cal E}_{P}$ in place of $l^{\cal E}_{!_P}$.

For any object $X$ of an elementary topos $\Etopos$, we shall denote, as in \cite{maclanemoerdijk}, by $ \{\cdot\}_X:X\rightarrow\powerset(X)$ the arrow whose mate $X\times X \to \Omega$ is the classifying arrow of the diagonal subobject $\Delta_X:X\rightarrowtail X\times X$. It operates as $x\mapsto\{x\}$ if $\Etopos=\Set$. 

In order to lighten notation, we shall sometimes omit subscripts $c$ for identity arrows $1_{c}$ when they can be unambigously inferred from the context.

To denote the Yoneda embedding $\cbicat\hookrightarrow[\cbicat\op,\Set]$ we will use the character $\yo$ (read ``yo''), the first ideogram of the name ``Yoneda'' in the hiragana Japanese alphabet.\footnote{For a reference on the usage of this notation, see \href{http://ncatlab.org/nlab/show/Yoneda+embedding\#ReferencesNotation}{Yoneda embedding in nLab}.}

\section{Dependent product in elementary toposes}
When working with sets the dependent product can be easily built by exploiting the equivalence $\Set/P\simeq \Set^P$: for a $P$-indexed set $B=\{B_p\}_{p\in P}$ the dependent product $\prod_f B$ is defined as $\{\prod_{f(p)=q} B_p \}_{q\in Q}$ (\cite[Theorem I.9.3]{maclanemoerdijk}). If we wish to go back to the slice topos $\Set/Q$, all we have to do is glue back the fibers into one single set. So, given $h:H\rightarrow P$, $\prod_f[h]$ is the set $\coprod_{q\in Q}\prod_{f(p)=q} h\inv(p)$,  with structural map the canonical projection to $Q$. 

We can attempt to generalize this construction to an arbitrary elementary topos by describing it with a suitable formula in the internal language. Notice that the above description contains set-indexed products and coproducts: as we shall see, what makes it possible to generalize it to any topos is the fact that all terms are ``bounded'', which allows us to represent them in terms of power objects and finitary products.  
 
For a set $W$ denote by $\powerset(W)$ its powerset: then for a family of sets $H_i \subseteq H$ indexed by $i \in I$ it holds that
\begin{align*}
\prod_{i\in I} H_i=\{& w\in \powerset(H\times I)\ |\ \forall\ i\in I\ \exists! \ x\in H ((x,i)\in w)\ \\&\hspace{2.1em}\AND\ \forall i\in I\ \forall\ x\in H ( (x,i)\in w\Rightarrow x\in H_i)   \}.
\end{align*}
In particular
\begin{align*}
\prod_{f(p)=q} h\inv(p)=\{& w\in \powerset(H\times f\inv(q))\ |\  \forall\ p\in f\inv(q)\ \exists!\ x\in H ((x,p)\in w), \\&\hspace{4.8em} \forall\ p\in f\inv(q)\ \forall x\in H ( (x,p)\in w \Rightarrow h(x)=p) )\}.
\end{align*}
Notice that the $w$ in the formula belongs to a set parametrized by $q$: since $f\inv(q)\subseteq P$, we can get rid of this dependence with the equivalent formulation
\begin{align*}
\prod_{f(p)=q} \hspace{-0.2em}h\inv(p)\hspace{-0.2em}=\hspace{-0.2em}\{& w\in \powerset(H\times P) \ |\ \forall p\in P (f(p)=q\Rightarrow \exists!\ x\in H ((x,p)\in w)),\\
&\hspace{2.9em} \forall\ p\in P\ \forall\ x\in H ((x,p)\in w\Rightarrow h(x)=p \AND f(p)=q ) \}.
\end{align*}
We can now glue all these fibers together along the indexing given by $Q$ and obtain the following expression for the dependent product:
\begin{prop}\label{prop.dprodint}
	Given $f:P\rightarrow Q$ and $h:H\rightarrow P$ in an arbitary topos $\Etopos$, the dependent product $\prod_f[h]$ is defined on objects by the internal language formula
	\begin{flalign*}
	\textstyle\prod_{f} [h]&=\coprod_{q\in Q}\prod_{p\in f\inv(q)} h\inv(p)\\
	&=\{ (q,w)\hspace{-0.1em}\in\hspace{-0.1em} Q\hspace{-0.1em}\times\hspace{-0.1em} \powerset(H\hspace{-0.1em}\times\hspace{-0.1em} P) \ |\ \forall p\hspace{-0.1em}\in\hspace{-0.1em} P\hspace{-0.1em} \ (f(p)\hspace{-0.1em}=\hspace{-0.1em}q\Rightarrow \exists!x\hspace{-0.1em}\in\hspace{-0.1em} H ((x,p)\hspace{-0.1em}\in\hspace{-0.1em} w)),\\
	&\hspace{7.6em} \forall p\in P\ \forall x\in H ((x,p)\in w\Rightarrow h(x)=p \AND f(p)=q )\}
	\end{flalign*}
	with, as structural morphism, the projection onto the $Q$-component.
\end{prop}
\begin{proof}
	The usual set-theoretic proof is constructive and therefore still works. Consider an arrow $k:K\rightarrow Q$: to an arrow $[k]\rightarrow \prod_f[h]$ given by $u\mapsto (k(u),\beta(u))\in \prod_f[h]\subseteq Q\times \powerset(H\times P)$ we associate the arrow $\alpha:f^*[k]\rightarrow [h]$ mapping $(p,u)$ to the unique $x$ such that $(p,x)\in\beta(u)$; vice versa, to $\alpha:f^*[k]\rightarrow [h]$ we associate the arrow $[k]\rightarrow\prod_f[h]$ defined by $u\mapsto (k(u),\{(p,x)\ |\ x=\beta(p,u) \})$. An easy computation shows that this is a natural equivalence. 
\end{proof}
This proof exploits the functional completeness of the internal language of a topos (\ie the possibility to define arrows element-wise) and the existence of power objects in a topos. Our aim now is to find a categorical description of the dependent product, inspired by the logical one just provided. 

Recall that, given a finitely complete category $\Etopos$ and an object $X\in \Etopos$, `the' \emph{power object} of $X$ is an object $\powerset(X)$ together with a monomorphism $\in_X\rightarrowtail X\times\powerset(X)$ such that for any subobject $n:N\rightarrowtail X\times Y$ there is a unique $n':Y\rightarrow \powerset(X)$ for which there is a pullback square 
	\begin{equation*}
		\begin{tikzcd}
		N \ar[r] \ar[d, tail, "n"'] & \in_X \ar[d, tail]\\
		X\times Y \ar[r, "1\times n'"] & X\times\powerset(X).
		\end{tikzcd}
	\end{equation*}
Notice that if $\Etopos$ is well-powered (that is, the collection of subobjects of any given object is a set) this means that $\powerset(X)$ is the representing object for the functor $\textup{Sub}(X\times -):\Etopos\op\rightarrow \Set$. In the sequel we will call $n'$ the \emph{classifying arrow} for $n$, and $n$ the \textit{classified subobject}. The subobject $\in_X$ is to be thought as the collection of pairs $(x,S)$ where $x\in X$, $S\subseteq X$ and $x\in S$: then $n'$ sends $y$ to $\{ x'\ |\ (x',y)\in N \}$, so that $N$ is indeed the pullback of $\in_X$ along $1\times n'$. 

The power-object construction can be made into a contravariant functor $\powerset:\Etopos\op\rightarrow\Etopos$: in the internal language, for any $\omega:X\rightarrow Y$ the arrow $\powerset(\omega):\powerset(Y)\to \powerset(X)$ operates as the inverse image. It is well known that a finitely complete category with all power objects is an elementary topos \cite[Sections A2.1-A2.3]{elephant}; in particular, the subobject classifier is $\Omega=\powerset(1_\Etopos)$. 

\begin{rmk}\label{rmk.powerinslice}
In $\Etopos/Q$, the power object of $[k:K\to Q]$ is the equalizer $$\Eq\left(\pi_{\powerset(K)}, \AND_K \circ ((\powerset(k)\circ \{\cdot\}_Q)\times 1_{\powerset(K)}): Q\times \powerset(K)\rightrightarrows \powerset(K)\right),$$
where $\AND_K:\powerset(K) \times \powerset(K) \to \powerset(K)$ is the intersection arrow, with the canonical projection onto $Q$ \cite[Theorem IV.7.1]{maclanemoerdijk}. Its domain is denoted by $\powerset_Q(k)$ and in the internal language is described as
$$\powerset_Q(k)=\{(q,S)\in Q\times \powerset(K)\ |\ S\subseteq k\inv(q) \}.$$	
\end{rmk}

Now, in order to obtain a categorical description of the object $\prod_f[h]$, we shall dissect its logical description given by Proposition \ref{prop.dprodint} into smaller pieces which will be more manageable. We can start from
$$S:=\{(p,w)\in P\times \powerset(H\times P)\ |\ \exists ! x\in H ( (x,p)\in w) \},$$
a subobject of $P\times\powerset(H\times P)$ which expresses functionality of $w$ in the variable $p$. To build it we consider the classifying arrow $\phi:P\times\powerset(H\times P)\rightarrow\powerset(H)$ for the subobject $e^{H}_{P}:=\in_{H\times P}\,\rightarrowtail H\times P\times\powerset(H\times P)$; this acts as $(p,w)\mapsto\{x\in H\ |\ (x,p)\in w\}$.
In the internal language, $S$ is the collection of pairs $(p,w)$ such that $\phi(p,w)$ is a singleton: in other words, $S$ is the following pullback: 
\begin{equation}\label{eq.defS}
\begin{tikzcd}
S\ar[d, tail] \ar[r] & H\ar[d, "{\{\cdot\}_H}", tail]\\
P\times\powerset(H\times P) \ar[r, "\phi"] & \powerset(H) \ar[ul, "\lrcorner"{near end, xshift=-2ex}, phantom]
\end{tikzcd}
\end{equation}

Now, in the above description of $\prod_f[h]$ this functionality of $w$ in $p$ is required for all the $p$'s in the fibers of $f$. To address this, we can exploit the (external) $\forall_g$ functor: in general, for a subobject $A\rightarrowtail X$ and an arrow $g:X\rightarrow Y$,
$$\forall_g(A)=\{y\in Y\ |\ g\inv(y)\subseteq A \}.$$
In particular, for $f\times 1_{\powerset(H\times P)}:P\times\powerset(H\times P)\rightarrow Q\times\powerset(H\times P)$, we obtain the object
$$\forall_{f\times 1}(S)\hspace{-0.1em}=\hspace{-0.1em}\{(q,w)\hspace{-0.1em}\in\hspace{-0.1em} Q\times \powerset(H\times P)\ |\ \forall p\hspace{-0.1em} \in\hspace{-0.1em} P\hspace{-0.1em} \ (f(p)=q \Rightarrow \exists ! x\hspace{-0.1em}\in\hspace{-0.1em} H ((x,p)\hspace{-0.1em}\in\hspace{-0.1em} w)) \}$$

Next, we should intersect $\forall_{f\times 1}(S)$ with
$$T^f_1=\{(q,w)\in Q\times \powerset(H\times P)\ |\ \forall\ p\in P\ \forall\ x\in H\ ((x,p)\in w\Rightarrow f(p)=q) \},$$
which expresses a ``fiber condition'' relating elements of $w$ and $q$: in fact, recall that $w$ should be thought as a tuple in $\prod_{p\in f\inv(q)} h\inv(p)$. The subobject $T^f_1$ can be described by the equivalent condition $w\subseteq \pi_P\inv f\inv(q)$ and thus coincides with $\powerset_Q(f\circ \pi_P:H\times P\rightarrow Q)$. But it is also equal to $\forall_\tau(W_1)$, where $\tau$ is the composite arrow $$\tau:Q\times\in_{H\times P}\,\rightarrowtail Q\times H\times P\times\powerset(H\times P) \rightarrow Q\times\powerset(H\times P)$$
and
\begin{equation*}
	\begin{tikzcd}
	W_1 \arrow[d, tail] \arrow[r, tail]               & H\times P\times\powerset(H\times P) \arrow[d, tail, "{\langle f\pi_P,1_{H\times P\times\powerset(H\times P)}\rangle}"] \\
	Q\times\in_{H\times P} \arrow[r, tail, "1_{Q}\times e^{H}_{P}"] & Q\times H\times P\times\powerset( H\times P), \arrow[lu, "\lrcorner"{near end, xshift=-3ex}, phantom]                
	\end{tikzcd}
\end{equation*}
where $\pi_{P}$ is the canonical projection $H\times P\times\powerset( H\times P) \to P$, is the subobject of quadruples $(q,x,p,w)$ such that $(x,p)\in w$ and $q=f(p)$.

Finally, for the last condition we intersect once more with
$$T^h_2=\{(q,w)\in Q\times \powerset(H\times P)\ |\ \forall\ p\in P\ \forall\ x\in H\ ((x,p)\in w\Rightarrow h(x)=p) \},$$
stating that whenever $(x,p)\in w$, $(x,p)$ also belongs to the graph of $h$. Similarly to $T^f_1$, we obtain that $T^h_2=\forall_\tau(W_2)$, where
\begin{equation*}
	\begin{tikzcd}
	W_2 \arrow[d, tail] \arrow[r, tail]               & Q\times H\times\powerset(H\times P) \arrow[d, tail, "{1_Q\times\langle 1_{H},h\rangle\times 1_{\powerset(H\times P)}}"] \\
	Q\times\in_{H\times P} \arrow[r, tail, "1_{Q}\times e^{H}_{P}"] & Q\times H\times P\times\powerset( H\times P) \arrow[lu, "\lrcorner"{near end, xshift=-3ex}, phantom]                
	\end{tikzcd}
\end{equation*}
is the subobject of quadruples $(q,x,p,w)$ such that $(x,p)\in w$ and $ p=h(x)$. 

Summarizing, we have shown that $\prod_f[h]\simeq \forall_{f\times1}(S)\cap T^f_1\cap T^h_2$.

We shall now give a purely categorical proof of the fact that the object $\forall_{f\times1}(S)\cap T^f_1\cap T^h_2$, with its canonical projection to $Q$, satisfies the universal property of the object $\prod_f[h]$. Everything reduces to the study of certain subobjects of $H\times P\times K$, which will serve as a connection between arrows $[f^*(k):f^{\ast}(K)\to P]\rightarrow[h]$ and $[k:K\to Q]\rightarrow\prod_f[h]$. Here is how they come into play: an arrow $\alpha: [f^*(k)]\rightarrow [h]$ is precisely an arrow $\alpha:f^{\ast}(K)\to H$ in $\Etopos$ satisfying the slice condition $h\circ \alpha=f^*(k)$; identifying this arrow with its graph $\langle \alpha,1\rangle:f^*(K)\rightarrowtail H\times f^*(K)$ and regarding $f^*(K)$ as a subobject of $P\times K$ via the monomorphism $\langle f^*(k),k^*(f)\rangle:f^*(K)\rightarrowtail P\times K$, we obtain a subobject $\langle m_H,m_P,m_K\rangle:M\rightarrowtail H\times P\times K$ satisfying the following properties:
	\begin{enumerate}[(i)]
		\item $h\circ m_H=m_P$, \ie it is a morphism in the slice topos $\Etopos\slash P$;
		\item there is an isomorphism $\bar{m}:M\xrightarrow{\sim} f^*(K)$ such that $f^*(k)\circ \bar{m}=m_P$ and $k^*(f)\circ \bar{m}=m_K$, \ie $M$ represents the graph of an arrow $f^*(K)\rightarrow H$.
	\end{enumerate}

In other words, the following diagram is commutative:
\begin{equation}\label{arrowfrompb}
\begin{tikzcd}[column sep=large]
 M\ar[d, "{\langle m_H,m_K\rangle}"'] \ar[r, "\sim", no head] \ar[dr, "{\langle m_H, m_P, m_K\rangle}"{description}, tail] & f^*(K) \ar[d,tail, "{\langle \alpha, f^*(k), k^*(f)\rangle}"]\\
 H\times K\ar[r, "{\langle 1_{H},h\rangle\times 1_{K}}"{below, yshift=-0.2ex}]  & H\times P\times K
\end{tikzcd}
\end{equation}

On the other hand, \textit{any} arrow $\beta: K\rightarrow\powerset(H\times P)$ corresponds to a subobject $M$ of $H\times P\times K$, which can be described, using the internal language, as the collection of triples $(x,p,u)$ such that $(x,p)\in \beta(u)$ and, categorically, as the following pullback:
\begin{equation}\label{Mfrombeta}
	\begin{tikzcd}[column sep=large]
	M \ar[d, "{\langle m_H,m_P,m_K\rangle}"', tail] \ar[r] & \in_{H\times P}  \ar[d, tail]\\
	H\times P\times K \ar[r, "1_H\times 1_P\times \beta"{below, yshift=-0.5ex}] & H\times P\times \powerset(H \times P)  \ar[ul, "\lrcorner"{near end, xshift=-3ex}, phantom]
	\end{tikzcd}
\end{equation}

\begin{lemma}
	Consider $\langle k,\beta\rangle: K\rightarrow Q\times\powerset(H\times P)$ and $M$ classified by $\beta$ as in square (\ref{Mfrombeta}); then
	\begin{enumerate}[(i)]\label{lemmatecnico}
		\item \label{lemmaT1} $\langle k,\beta\rangle$ factors through $T^f_1$ if and only if $\langle m_H,m_P,m_K\rangle:M\rightarrowtail H\times P\times K$ factors through $1_H\times \langle f^*(k),k^*(f)\rangle: H\times f^*(K)\rightarrowtail H\times P\times K$;
		\item \label{lemmaT2} $\langle k,\beta\rangle$ factors through $T^h_2$ if and only if $h\circ m_H=m_P$;
		\item \label{lemmaS} 	$\langle k, \beta\rangle$ factors through $\forall_{f\times1}(S)$ if and only if there is a morphism $\alpha: f^*(K)\rightarrow H$ such that $\langle \alpha,1\rangle:f^*(K)\rightarrowtail H\times f^*(K)$ is the pullback of $M$ along $1_H\times\langle f^*(k),k^*(f)\rangle$.
	\end{enumerate}
	
\end{lemma}
\begin{proof} 
	We preliminarily observe that we have the following rectangle, whose internal squares are both pullbacks and whose lower composite arrow is $\tau$: 
\begin{equation*}
	\begin{tikzcd}
	M \arrow[r, "{\langle m_H,m_P,m_K\rangle}", tail] \arrow[d] & H\times P\times K \arrow[r] \arrow[d, "{\langle k\circ \pi_K, 1_H\times 1_P\times \beta\rangle}"'] & K \arrow[d, "{\langle k,\beta\rangle}"] \\
	Q\times\in_{H\times P} \arrow[r, "1_{Q}\times e^{H}_{P}"{below, yshift=-1ex}, tail]                      & Q\times H\times P\times\powerset(H\times P)  \arrow[lu, "\lrcorner"{near end, xshift=-2.8ex}, phantom] \ar[r]                      & Q\times\powerset(H\times P) \arrow[lu, "\lrcorner"{near end, yshift=0.6ex, xshift=-1.8ex}, phantom]           
	\end{tikzcd}
\end{equation*}	
\begin{enumerate}[(i)]
	\item $\langle k,\beta\rangle$ factors through $T^f_1=\forall_\tau(W_1)$ if and only if its pullback along $\tau$ factors through $W_{1}\rightarrowtail Q\times\in_{H\times P}$. By the universal property of the pullback square defining $W_{1}$, $\tau^{\ast}(\langle k,\beta\rangle)$ factors through $W_{1}\rightarrowtail Q\times\in_{H\times P}$ if and only if the composite arrow ${\langle k\circ \pi_K, 1_H\times 1_P\times \beta\rangle}\circ {\langle m_H,m_P,m_K\rangle}=\langle k\circ  m_K, m_H, m_P, \beta \circ m_K\rangle: M\rightarrow Q\times H\times P\times\powerset(H \times P)$ factors through $\langle f\circ \pi_P,1_{P\times \powerset(H\times P)}\rangle:H\times P\times\powerset(H\times P)\rightarrowtail Q\times H\times P\times \powerset(H\times P)$. Now, if such a factorization exists then it is necessarily equal to $\langle m_H, m_P, \beta \circ m_K\rangle:M\to H\times P\times\powerset( H\times P)$, and this arrow satisfies the required property if and only if, denoting by $\pi_{Q}$ the canonical projection $Q\times H\times P\times\powerset( H\times P) \to Q$, $\pi_{Q}\circ \langle f\circ \pi_P,1_{P\times \powerset(H\times P)}\rangle\circ \langle m_H, m_P, \beta \circ m_K\rangle=\pi_{Q}\circ \langle k\circ  m_K, m_H, m_P, \beta \circ m_K\rangle$. But this holds precisely when $f\circ m_P=k\circ m_K$, \ie $\langle m_H, m_P,m_K\rangle$ factors through $1_H\times \langle f^*(k),k^*(f)\rangle: H\times f^*(K)\rightarrowtail H\times P\times K$.

	\item Similarly to (i), $\langle k,\beta\rangle$ factors through $T^h_2=\forall_\tau(W_2)$ if and only if its pullback along $\tau$ factors through $W_{2}\rightarrowtail Q\times\in_{H\times P}$, equivalently if and only if the arrow  $\langle k\circ  m_K, m_H, m_P, \beta \circ m_K\rangle$ factors through $1_{Q}\times\langle 1_{H},h\rangle\times 1_{\powerset(H\times P)}:Q\times H\times\powerset(H\times P)\rightarrow Q\times H\times P\times \powerset(H\times P)$; now, if such a factorization exists it is necessarily equal to $\langle k\circ  m_K, m_H, \beta \circ m_K\rangle$, and this arrow satisfies the required property if and only if, denoting by $\pi'_{P}$ the canonical projection $Q\times H\times P\times\powerset( H\times P) \to P$, $\pi'^{P}\circ \langle k\circ  m_K, m_H, m_P, \beta \circ m_K\rangle=\pi'^{P}\circ (1_{Q}\times\langle 1_{H},h\rangle\times 1_{\powerset(H\times P)}) \circ \langle k\circ  m_K, m_H, \beta \circ m_K\rangle$, \ie $m_P=h\circ m_H$.
	
	\item $\langle k, \beta\rangle$ factors through $\forall_{f\times 1}(S)$ if and only if its pullback along $f\times 1$, \ie $\langle f^*(k), \beta \circ k^*(f)\rangle: f^*(K)\rightarrow P\times\powerset(H\times P)$, factors through $S\rightarrowtail P\times\powerset(H\times P)$; this happens if and only if there is some $\alpha: f^*(K)\rightarrow H$ such that $\{\cdot\}_H \circ \alpha=\phi \circ \langle f^*(k),\beta \circ k^*(f)\rangle$ ($\alpha$ is also unique, since $\{\cdot\}_H$ is monic: see \cite[Corollary A2.2.3]{elephant}). Now, the arrows $\{\cdot\}_H \circ \alpha$ and $\phi \circ \langle f^*(k),\beta \circ k^*(f)\rangle$ are equal if and only if they classify the same subobject of $H\times f^*(K)$. It is immediate to see that $\{\cdot\}_H\circ \alpha$ classifies $\langle \alpha,1_{f^{\ast}(K)}\rangle:f^*(K)\rightarrowtail H\times f^*(K)$, and that $\phi \circ \langle f^*(k),\beta \circ k^*(f)\rangle$ classifies the pullback of $\in_{H\times P}\,\rightarrowtail H\times P\times\powerset(H\times P)$ along $1_{H}\times \langle f^*(k),\beta \circ k^*(f)\rangle$, which coincides (by the pullback lemma) with the pullback of $\langle m_H,m_P,m_K\rangle$ along $1_{H}\times \langle f^*(k),k^*(f)\rangle$. So $\{\cdot\}_H \circ \alpha=\phi \circ \langle f^*(k),\beta \circ k^*(f)\rangle$ if and only if the left-hand square in the following diagram is a pullback:
	\begin{equation*}
		\begin{tikzcd}[column sep=large]
		f^*(K)\ar[r, tail]\ar[d, tail, "{\langle\alpha,1_{f^{\ast}(K)}\rangle}"']&M \ar[d, "{\langle m_H,m_P,m_K\rangle}"', tail] \ar[r] & \in_{H\times P} \ar[d, tail] \\
		H\times f^*(K)\ar[r, tail, "{1_H\times\langle f^*(k),k^*(f)\rangle}"{below, yshift=-1ex}] &H\times P\times K \ar[r, "1_H\times 1_P\times \beta"{below, yshift=-1ex}] & H\times P\times \powerset(H \times P) \ar[ul, "\lrcorner"{near end, xshift=-3ex}, phantom]
		\end{tikzcd}
	\end{equation*}
\end{enumerate}\vspace{-2ex}\end{proof}
Notice that the factorization of $\langle k,\beta\rangle$ through $\forall_{f\times 1}(S)$ alone grants the existence of $\alpha$. Using the internal language, we can express this condition as the requirement that for every $p\in P$, if $f(p)=k(u)$ then there exists a unique $x$ such that $(x,p)\in \beta(u)$; in fact, $\alpha:f^*(K)\rightarrow H$ assigns to each such pair $(p,u)$ that single $x\in H$. On the other hand, $T^f_1$ and $T^h_2$ provide the fiber-like conditions which $\alpha$ and $\beta$ must satisfy.

We are now ready to state the central result of this section:
\begin{thm} Let $h:H\to P$ be an object of $\Etopos \slash P$. Then, with the above notation, $\prod_f[h]\cong \forall_{f\times1}(S)\cap T^f_1\cap T^h_2$. More specifically, for any object $k:K\to Q$ of $\Etopos \slash Q$, there is a natural bijective correspondence between the arrows $[f^*(k):f^{\ast}(K)\to P]\rightarrow[h]$ in $\Etopos \slash P$ and $[k:K\to Q]\rightarrow\prod_f[h]$ in $\Etopos \slash Q$. This correspondence sends 
	\begin{itemize}
		\item an arrow $\alpha:[f^*(k)]\rightarrow[h]$ in $\Etopos \slash P$ to the arrow $\langle k,\beta\rangle:[k]\to \forall_{f\times1}(S)\cap T^f_1\cap T^h_2\rightarrowtail Q\times\powerset(H\times P)$ in $\Etopos\slash Q$, where $\beta:K\to \powerset(H\times P)$ is the classifying arrow of the graph of $\alpha$, regarded as a subobject of $H\times P\times K$;
		
		\item an arrow $\langle k,\beta\rangle:[k]\to \forall_{f\times1}(S)\cap T^f_1\cap T^h_2\rightarrowtail Q\times\powerset(H\times P)$ in $\Etopos\slash Q$ to the arrow $\alpha:[f^*(k)]\rightarrow[h]$ in $\Etopos \slash P$ whose graph in $\Etopos$ is the subobject of $H\times P\times K$ classified by $\beta$.
	\end{itemize}
\end{thm}
\begin{proof}
	We will show that a subobject $\langle m_H,m_P,m_K\rangle:M\rightarrowtail H\times P\times K$ makes diagram (\ref{arrowfrompb}) commutative, \ie corresponds to an arrow $\alpha:[f^*(k)]\rightarrow[h]$ in $\Etopos\slash P$, if and only if its classifying arrow $\beta:K\rightarrow\powerset(H\times P)$ is such that $\langle k,\beta\rangle $ factors through $\forall_{f\times 1}(S)\cap T^f_1\cap T^h_2$:
	\begin{equation*}
		\begin{tikzcd}
		f^*(K) \arrow[rdd, "f^*(k)"'] \arrow[rrd, "\alpha", dashed] \arrow[rrr, "k^*(f)"] &                     &                   & K \arrow[rdd, "k"'] \arrow[rrd, "{\langle k,\beta\rangle}", dashed] &   &                                                       \\
		&                     & H \arrow[ld, "h"] &                                                                     &   & \forall_{f\times1}(S)\cap T^f_1\cap T^h_2 \arrow[d, tail]\ar[dl] \\
		& P \arrow[rrr, "f"{yshift=-0.2ex, below}] &                   &                                                                     & Q & Q\times\powerset(H\times P) \arrow[l, "\pi_Q"{yshift=-0.5ex}]       
		\end{tikzcd}
	\end{equation*}
	 Lemma \ref{lemmatecnico}(\ref{lemmaT2}) says that $\langle k,\beta\rangle$ factors through $T^h_2$ if and only if $M$ satisfies condition (i) for diagram (\ref{arrowfrompb}). Lemma \ref{lemmatecnico}(\ref{lemmaT1}) tells us that $\langle k,\beta\rangle$ factors through $T^f_1$ if and only if $\langle m_H,m_P,m_K\rangle:M\rightarrowtail H\times P\times K$ factors through $1_H\times \langle f^*(k),k^*(f)\rangle: H\times f^*(K)\rightarrowtail H\times P\times K$, while by Lemma \ref{lemmatecnico}(\ref{lemmaS}) $\langle k,\beta\rangle$ factors through $\forall_{f\times 1}(S)$ if and only if its pullback along $1_H\times \langle f^*(k),k^*(f)\rangle: H\times f^*(K)\rightarrowtail H\times P\times K$ is isomorphic to $\langle \alpha,1\rangle:f^*(K)\rightarrowtail H\times f^*(K)$. Therefore, $\beta$ factors through $T^f_1$ and $\forall_{f\times 1}(S)$ if and only if condition (ii) for diagram (\ref{arrowfrompb}) is satisfied.
	 So we can conclude that  $\beta:K\rightarrow \forall_{f\times 1}(S)\cap T^f_1\cap T^h_2$ as a morphism of $\Etopos/Q$ corresponds to a unique morphism $\alpha:[f^*(k)]\rightarrow[h]$ in $\Etopos\slash P$ and viceversa. The naturality of this correspondence is immediate, as all the arrows involved in it are defined by universal properties.
\end{proof}	

One might complain that the pervasive appearance of $\forall$ is not a big step towards an elementary treatment of the dependent product: yet it can be reduced to more basic structures, as we will show in a moment. To do so we must recall the definition of the \emph{covariant power-object functor} (see \cite[pag. 92]{elephant}): it is defined as $\powerset$ on objects, but it sends an arrow $f: Y\rightarrow X$ to the arrow $\exists f: \powerset(Y)\rightarrow\powerset(X)$ classifying the image of $\in_Y \, \rightarrowtail Y\times\powerset(Y)\xrightarrow{f\times 1} X\times \powerset(Y)$. The notation $\exists f$ is justified by the fact that $S\subseteq Y$ is sent to $\{f(y) \, |\ y\in S \}=\{ x\in X\ |\ \exists y\in S \ ( f(y)=x) \}$. Using the covariant power object functor, $\forall$ can be described in an elementary way. Indeed, the following result provides a description of the action of $\forall$ on subobjects entirely in terms of finite limits and power objects: 
\begin{prop}
	Let $f:Y\rightarrow X$ be an arrow and $i:A\rightarrowtail Y$ a subobject in an elementary topos $\Etopos$. Then $\forall_f(A)\simeq A'$ (as subobjects of $X$), where $A'$ is defined by the following pullback square:
	\begin{center}
		\begin{tikzcd}
		A' \ar[d, tail] \ar[rr]& & \powerset(A)\ar[d, rightarrowtail,"\exists i"]\\
		X\ar[r, "\{\cdot\}_X"'] & \powerset(X)\ar[r, "\powerset(f)"'] &\powerset(Y)\ar[ull, "\lrcorner"{xshift=-3ex, near end}, phantom]
		\end{tikzcd}
	\end{center}
More explicitly, the vertical arrow $\exists i:\powerset(A) \to \powerset(Y)$ is the classifying arrow of the composite subobject 
\begin{equation*}
\begin{tikzcd}
\in_{A} \ar[r, tail] \ar[r] & A\times \powerset(A)  \ar[r, tail, "i \times 1_{\powerset(A)}"] & Y\times \powerset(A) 
\end{tikzcd}
\end{equation*}
and the horizontal arrow $\powerset(f)\circ \{\cdot\}_X:X\to \powerset(Y)$ is the classifying arrow of the graph of $f$.
\end{prop}
\begin{proof}
		Using the internal language, we have:
		$$A'\hspace{-0.2em}=\hspace{-0.2em}\{(x,N)\hspace{-0.1em}\in\hspace{-0.1em} X\times\powerset(A)\ |\ f^{-1}(x)\hspace{-0.1em}=\hspace{-0.1em}i(N)\}\hspace{-0.2em}\cong\hspace{-0.2em}\{x\hspace{-0.1em}\in\hspace{-0.1em} X\ |\ f\inv(x)\hspace{-0.1em}\subseteq\hspace{-0.1em} i(A) \}\hspace{-0.1em}=\hspace{-0.1em}\forall_f(A).$$
\end{proof}

Distinguishing $T^f_1$ and $T^h_2$ we were able to isolate what exactly causes the induced arrow $\alpha:f^*(K)\rightarrow H$ to yield a morphism in the slice topos $\Etopos\slash P$; yet, we could have treated the intersection of $T^f_1$ and $T^h_2$ as a whole from the very beginning: being both $T^f_1=\forall_\tau(W_1)$ and $T^h_2=\forall_\tau(W_2)$, their intersection is $\forall_\tau(W_1\cap W_2)$ ($\forall$ commutes with intersections as it is a right adjoint). But there is an alternative and very simple formulation for $T^f_1\cap T^h_2$ in terms of a power object in the topos $\Etopos \slash Q$:
\begin{prop}
	As a subobject of $Q\times\powerset(H\times P)$, $T^f_1\cap T^h_2$ is isomorphic to the composite monomorphism $$\powerset_Q(H\xrightarrow{f\circ h}Q)\rightarrowtail Q\times\powerset(H)\xrightarrowtail{1_{Q}\times\exists(\langle1,h\rangle)}Q\times\powerset(H\times P).$$
\end{prop}
\begin{proof}
	Let us argue by using the internal language. We have
	$\powerset_Q(H\xrightarrow{f\circ h}Q)=\{(q,S)\in Q\times \powerset(H)\ |\ S\subseteq (f\circ h)\inv(q) \}$ by Remark \ref{rmk.powerinslice}. So $(1_{Q}\times\exists(\langle1,h\rangle))\circ \powerset_Q(H\xrightarrow{f\circ h}Q)=\{(q, w) \in Q\times \powerset(H\times P) \mid (\exists S)((q, S)\in \powerset_Q(H\xrightarrow{f\circ h}Q) \wedge w= \langle1,h\rangle(S))\}=\{(q, w)\in Q\times \powerset(H\times P) \mid \forall (x, p)\in w \!\ (p=h(x) \wedge f(p)=q)\}=T^f_1\cap T^h_2$. 
\end{proof}

To conclude the section, we describe what happens to the dependent product functor when we restrict it to a subtopos $\Ftopos$ of $\Etopos$. In general, any geometric morphism $\psi:\Ftopos\rightarrow \Etopos$ induces, for each $E\in \Etopos$, a geometric morphism $\psi_E:\Ftopos/\psi^*(E) \to \Etopos/E$ making the following diagram commutative:
\begin{equation*}
	\begin{tikzcd}
	\Ftopos \ar[r, "\psi"] & \Etopos\\
	\Ftopos/\psi^*(E)\ar[u, "l^{\Ftopos}_{\psi^*(E)}"] \ar[r, "\psi_E"'] & \Etopos/E \ar[u, "l^{\Etopos}_{E}"'] 
	\end{tikzcd}
\end{equation*}
The inverse image $\psi^*_E$ is defined as $\psi_E^*([h]):=[\psi^*(h)]$, while the direct image $\psi_{E*}$ sends an object $k:K\rightarrow \psi^*(E)$ to the pullback of $\psi_*(k)$ along the unit $\eta_E:E\rightarrow \psi_*\psi^*(E)$ \cite[Example A4.1.3]{elephant}; in particular, when $\psi$ is an inclusion $\psi_E$ is also an inclusion \cite[Example 5.18]{caramello.denseness}. Let us now consider a subtopos $i:\Ftopos \hookrightarrow \Etopos$. For any arrow $f:P\to Q$, the square 
 \begin{equation*}
 	\begin{tikzcd}
 	\Etopos/P \ar[d, "(i_{P})^{\ast}"']  &\Etopos/Q \ar[l, "f^*"'] \ar[d, "(i_Q)^{\ast}"]\\
 	\Ftopos/i^{\ast}(P)   &\Ftopos/i^{\ast}(Q) \ar[l, "{(i^{\ast}(f))}^*"] 
 	\end{tikzcd}
\end{equation*} 
of inverse images is commutative, since $i^{\ast}$ preserves pullbacks. Therefore, the square of direct images if also commutative. Notice that if $P$ and $Q$ lie in $\Ftopos$, $i^{\ast}(P)\cong P$, $i^{\ast}(Q)\cong Q$ and $(i_P)_{\ast}$, $(i_Q)_{\ast}$ are the canonical inclusion functors induced by the embedding of $\Ftopos$ into $\Etopos$.

Summarizing, we have the following result:

\begin{prop}\label{prop.sottotopos}
	Let $i:\Ftopos\hookrightarrow \Etopos$ be a subtopos and $f:P\rightarrow Q$ an arrow in $\Etopos$. Then
	\begin{equation}
		\begin{tikzcd}
		\Etopos/P \ar[r, "\prod^{\Etopos}_f"] &\Etopos/Q\\
		\Ftopos/i^{\ast}(P) \ar[u, "(i_P)_{\ast}"] \ar[r, "\prod^{\Ftopos}_{i^{\ast}(f)}"'] &\Ftopos/i^{\ast}(Q) \ar[u, "{(i_Q)}_{\ast}"']
		\end{tikzcd}
	\end{equation} 
	is a commutative diagram of geometric morphisms. In particular, if $f:P\rightarrow Q$ is a morphism in $\Ftopos$, then $\prod_f^\Ftopos:\Ftopos/P\rightarrow\Ftopos/Q$ is the restriction of $\prod_f^\Etopos:\Etopos/P\rightarrow\Etopos/Q$ along the canonical inclusions $\Ftopos/P \hookrightarrow \Etopos/P$ and $\Ftopos/Q \hookrightarrow \Etopos/Q$.
\end{prop} 

\section{Dependent product for Grothendieck toposes}

In this section we shall provide an explicit site-level description of dependent products in a Grothendieck topos. This relies on the fact that the slices $\Etopos\slash P$ and $\Etopos\slash Q$ of our topos $\Etopos$ can be represented as toposes of sheaves on suitable sites in such a way that the geometric morphism
$$l^{\cal E}_{f}:\Etopos\slash P \to \Etopos \slash Q$$
is induced by a comorphism between them.

Given a subtopos $\sheafify\dashv i:\Ftopos \hookrightarrow \Etopos$ and an object $E$ of $\Etopos$, we shall denote the direct image of the canonical geometric morphism $\Ftopos \slash \sheafify(E) \to \Etopos \slash E$ by $i_{E}$ and its inverse image by $\sheafify_{E}$.

Recall that the category of elements $\int P$ of a presheaf $P:\cbicat\op\rightarrow\Set$ has as objects the pairs $(X,x)$ with $X\in \cbicat$ and $x\in P(X)$ and as morphisms $(Y,y)\rightarrow (X,x)$ the arrows $g:Y\to X$ such that $P(g)(x)=y$. Categories of elements are in fact presentation sites for slice toposes: 
\begin{prop}\label{prop.sitoperslice}
Consider a sheaf topos $\sheafify\dashv i:\textup{Sh}(\cbicat,J)\hookrightarrow [\cbicat \op, \Set]$ and a presheaf $P:\cbicat\op\rightarrow\Set$. There is an equivalence of toposes 
$$R^J_P:\textup{Sh}(\cbicat, J)/\sheafify(P)\stackrel{\simeq}{\to} \textup{Sh}(\textstyle \int P, J_P),$$ where $J_P$ is the Grothendieck topology whose sieves are precisely those sent to $J$-covering sieves by the canonical functor $\pi_P:\int P\rightarrow\cbicat$. 
\begin{itemize}
	\item In the context of presheaves, $R_P:[\cbicat\op,\Set]/P\rightarrow [\left(\int P\right)\op,\Set]$ is defined by $$R_P[h](X,p)=h_X\inv(p).$$ Its pseudoinverse $L_P:[\left(\int P\right)\op,\Set]\rightarrow [\cbicat\op,\Set]/P$ is defined for a presheaf $W:(\int P)\op\rightarrow\Set$ as $$L_P({W})(X):=\bigsqcup_{p\in P(X)} W(X,p),$$ with structural morphism to $P$ the projection of each component $W(X,p)$ to the element $p$ indexing it. The definition on $R_P$ and $L_P$ on arrows is straightforward.
	
	\item The functor $R_P^J:\textup{Sh}(\cbicat,J)/\sheafify(P)\rightarrow \textup{Sh}(\int P, J_P)$ acts as $R_P\circ i_P$, that is $$R_P^J[w](X,p)\hspace{-0.2ex}=R\hspace{-0.2ex}_P[\eta_P^*w](X,p)\hspace{-0.2ex}\cong\hspace{-0.2ex}\{ x\hspace{-0.2ex}\in\hspace{-0.2ex} W(X)\ |\ w_X(x)\hspace{-0.2ex}=\hspace{-0.2ex}(\eta_P)_{X}(p) \},$$ 
	where $\eta_P:P\to a(P)$ is the unit of the adjunction $\sheafify \dashv i$. 
	
	Its pseudoinverse $L_P^J=\textup{Sh}(\int P, J_P) \to \textup{Sh}(\cbicat,J)/\sheafify(P)$ can be expressed as $L_P^J= \sheafify_P \circ L_P \circ (i_{J_P})_{\ast}$ (where $i_{J_P}$ is the canonical geometric inclusion $\textup{Sh}(\int P, J_P)\hookrightarrow [\left(\int P\right)\op,\Set]$), \ie
	$$L_P^J(W)(X)\hspace{-0.1em}:=\hspace{-0.1em}\colim(\textstyle\int W\rightarrow \textstyle\int P\rightarrow \textup{Sh}(\cbicat,J)/\sheafify(P))=\hspace{-0.3em}\underset{x\in W(X,p)}{\colim}\sheafify(\yo(X))$$
	and the structural morphism is the one induced by the $\sheafify(\tilde{p})$'s, where $\tilde{p}:\yo(X)\rightarrow P$ corresponds to $p\in P(X)$ via the Yoneda lemma:
	\begin{equation}\label{isoslice}
	\begin{tikzcd}
	{[\cbicat\op,\Set]/P} \ar[d, "\sheafify_P"', "\dashv"{}, xshift=-1ex]\arrow[r, "R_P", yshift=1ex]         \ar[r, phantom, "\sim"{yshift=-0.25ex}]                          & {[\left(\int P\right)\op,\Set]}      \ar[l, "L_P", yshift=-1ex]      \ar[d, "(i_{J_P})^{\ast}"', "\dashv"{}, xshift=-1ex]  \\
	{\textup{Sh}(\cbicat, J)/\sheafify(P)} \arrow[r, "R_P^J", yshift=1ex] \arrow[u, "i_P"', hook, xshift=1ex]\ar[r, phantom, "\sim"{yshift=-0.25ex}]  & {\textup{Sh}(\int P, J_P)} \arrow[u, "(i_{J_P})_{\ast}"', hook, xshift=1ex] \ar[l, "L_P^J", yshift=-1ex]
	\end{tikzcd}
	\end{equation}
\end{itemize}
\end{prop}
\begin{proof} 
	Let $i:\textup{Sh}({\cal C}, J)\hookrightarrow [{\cal C}^{\textup{op}}, \Set]$ be the canonical geometric inclusion and $\sheafify$ the associated sheaf functor $[{\cal C}^{\textup{op}}, \Set] \to \textup{Sh}({\cal C}, J)$.
	By \cite[Section 5.7]{caramello.denseness}, the geometric equivalence $[\cbicat\op,\Set]/P\simeq [\left(\int P\right)\op,\Set]$ given by the functors $L_P \dashv R_{P}$ restricts, along the geometric inclusions $$i_{P}:\textup{Sh}({\cal C}, J)\slash a_{J}(P)\hookrightarrow [{\cal C}^{\textup{op}}, \Set]\slash P$$ and $$i_{J_{P}}:\textup{Sh}(\textstyle \int P, J_P)\hookrightarrow [\left(\textstyle \int P\right)\op,\Set],$$ to a geometric equivalence $L^J_P \dashv R^J_{P}$. From this it follows at once that $R^J_P$ acts as $R_P\circ i_P$, and that $L_P^J= \sheafify_P \circ L_P \circ (i_{J_P})_{\ast}$. 
\end{proof}

\begin{rmk}\label{rmkrestrictiontosheaves}
	It can be easily checked that if $P$ is a $J$-sheaf then the functor $L_{P}:[\left(\int P\right)\op,\Set] \to [\cbicat\op,\Set]/P$ takes values in $J$-sheaves (Exercise III.8(b) \cite{maclanemoerdijk}, whence the functor $L_{P}^{J}$ is the restriction of $L_{P}$ (along the canonical inclusions ${\textup{Sh}(\int P, J_P)}\hookrightarrow [\left(\int P\right)\op,\Set]$ and $\textup{Sh}({\cal C}, J)/P \hookrightarrow [\cbicat\op,\Set]/P$). Of course, also $R^J_{P}$ is the restriction of $R_P$. 
\end{rmk}

We recall that a comorphism of sites $(\cbicat, J)\to (\dbicat, K)$ \cite[p. 574]{elephant} is a functor $F:\cbicat\rightarrow\dbicat$ satisfying the \emph{covering-lifting property}, \ie such that for every sieve $R\in K(F(X))$ there exists some sieve $S\in J(X)$ such that $F(S)\subseteq R$. Comorphisms of sites are precisely those functors $F$ such that the geometric morphism $-\circ F\op\dashv \ran_{F\op}:[\cbicat\op,\Set] \rightarrow [\dbicat\op,\Set]$ restricts to sheaves: in particular, there is a geometric morphism $C(F):\textup{Sh}(\cbicat,J) \rightarrow \textup{Sh}(\dbicat,K) $ such that $i_K \circ C(F)_*=\ran_{F\op} \circ i_J$. Its inverse image is $C(F)^*=\sheafify_J \circ  (-\circ F\op)\circ i_K$.

Now, any $f:P\rightarrow Q$ in $[\cbicat\op,\Set]$ induces a functor $\int f:\int P\rightarrow \int Q$ acting on objects as $({\int f})(X,p)=(X, f_X(p))$ and on arrows as $({\int f})(\phi)=\phi$. The functor $\int f$ is a comorphism of sites $(\int P, J_P)\to (\int Q, J_Q)$. Indeed, consider $(X,p)\in \int P$ and $R\in J_Q((X,f_X(p)))$. Since $\pi_{Q}(R)$ is a $J$-covering sieve, there is a $J_P$-covering sieve on $(X,p)$ formed by arrows $\phi:(Y, P(\phi)(p))\rightarrow (X,p)$ where $\phi\in \pi_{Q}(R)$, and its image through $\int f$ is precisely $R$. Therefore, we have a commutative diagram of direct image functors:
\begin{equation}\label{comorphism}
	\begin{tikzcd}[column sep=large]
	{[\left(\int P\right)\op,\Set]} \arrow[r, "\ran_{(\int f)\op}"]           & {[\left(\int Q\right)\op,\Set]}          \\
	{\textup{Sh}(\int P, J_P)} \arrow[r, "C(\int f)_{\ast}"] \arrow[u, "(i_{J_P})_{\ast}", hook] & {\textup{Sh}(\int Q, J_Q)} \arrow[u, "(i_{J_Q})_{\ast}"', hook]
	\end{tikzcd}
\end{equation}
This square is actually induced by the square 
\begin{center}
	\begin{tikzcd}
	(\int P, T_{P}) \ar[r, "\int f"] & (\int Q, T_{Q})\\
	(\int P, J_P) \ar[u, "1_{\int P}"] \ar[r, "\int f"] & (\int Q, J_Q) \ar[u, "1_{\int P}"']
	\end{tikzcd}
\end{center}
of comorphisms of sites, where $T_{P}$ (resp. $T_{Q}$) is the trivial Grothendieck topology on $\int P$ (resp. $\int Q$). 

\begin{thm}\label{dp.groth}
	The direct image functor $C(\int\hspace{-0.1em} f)_{\ast}:\textup{Sh}(\int\hspace{-0.2em} P, J_P)\rightarrow \textup{Sh}(\int\hspace{-0.1em} Q, J_Q)$ of the geometric morphism $C(\int f)$ corresponds to the dependent product $ \prod_{\sheafify(f)}:\textup{Sh}(\cbicat,J)/\sheafify(P)\rightarrow \textup{Sh}(\cbicat,J)/\sheafify(Q)$ via the equivalences of Proposition \ref{prop.sitoperslice}:
	\begin{equation*}
	\begin{tikzcd}
	\textup{Sh}(\int P, J_P) \arrow[r, "C(\int f)_{\ast}"] \ar[d, "L_P^J", xshift=1ex] \ar[d, phantom, "\sim"{xshift=+0.25ex, rotate=90}]  & \textup{Sh}(\int Q, J_Q)  \ar[d, "L_Q^J", xshift=1ex] \ar[d, phantom, "\sim"{xshift=+0.25ex, rotate=90}]  \\
	{\textup{Sh}(\cbicat, J)/\sheafify(P)} \arrow[u, "R_P^J", xshift=-1ex] \arrow[r, "\prod_{\sheafify(f)}"]  & {\textup{Sh}(\cbicat, J)/\sheafify(Q)} \ar[u, "R_Q^J", xshift=-1ex]
	\end{tikzcd}
	\end{equation*}
\end{thm}
\begin{proof}
Let us first consider the presheaf case. For any $W:\left(\int Q\right)\op\rightarrow\Set$, $(f^*\circ L_Q)(W)(X)\cong \bigsqcup _{p\in P(X)} W(X,f_X(p))$ $=L_P \circ (-\circ(\int f)\op)(W)(X)$, naturally in $X\in {\cal C}$. Therefore $L_P\circ (-\circ (\int f)\op)=f^* \circ L_Q$ and so 
\begin{equation}\label{dp.psh}
R_Q \circ \textstyle\prod_f=\ran_{(\int f)\op} \circ R_P.
\end{equation} 

The more general case amounts precisely to the commutativity of the front square in the following cube:
\begin{center}
	\begin{tikzcd}[row sep=large, column sep=small]
	& {[\left(\int\hspace{-0.2ex} P\right)\op\hspace{-0.2ex},\Set]} \arrow[dd, "R_P"{yshift=4ex}, leftarrow] \arrow[rr, "\ran_{(\int\hspace{-0.2ex} f)\op}"] &                                                                             & {[\left(\int\hspace{-0.2ex} Q\right)\op\hspace{-0.2ex},\Set]} \arrow[dd, "R_Q"{yshift=4ex}, leftarrow] \\
	{\textup{Sh}(\int\hspace{-0.2ex} P, J_P)} \arrow[dd, "R_P^J"{yshift=2.5ex, xshift=-3.8ex}, leftarrow] \arrow[rr, "C(\int\hspace{-0.2ex} f)_*"{xshift=-4ex}, crossing over] \arrow[ru, "(i_{J_P})_*", hook] &                                                                                             & {\textup{Sh}(\int\hspace{-0.2ex} Q, J_Q)}  \arrow[ru, "(i_{J_Q})_*", hook] &                                                            \\
	& {[\cbicat\op,\Set]/P} \arrow[rr, "\prod_f"{below, xshift=-5ex}]                                                &                                                                             & {[\cbicat\op,\Set]/Q}                                      \\
	{\textup{Sh}(\cbicat,J)/\sheafify(P)} \arrow[ru, "i_P"', hook] \arrow[rr, "\prod_{\sheafify(f)}"']             &                                                                                             & {\textup{Sh}(\cbicat,J)/\sheafify(Q)} \arrow[ru, "i_Q"', hook]        \arrow[uu, "R_Q^J"{yshift =2.5ex}, crossing over]               &                                                           
	\end{tikzcd}
\end{center} 
This in turn follows from Equation (\ref{dp.psh}), Square (\ref{isoslice}), Square (\ref{comorphism}) and Proposition \ref{prop.sottotopos}.
\end{proof}

\begin{cor}\label{prop.ranpi}
	The direct image $C(\pi_{P})_*:\textup{Sh}(\int P, J_P)\rightarrow \textup{Sh}(\cbicat, J)$ of the geometric morphism $C(\pi_{P})$ corresponds to the dependent product functor $\prod_{\sheafify(P)}:\textup{Sh}(\cbicat, J)/\sheafify(P)\rightarrow \textup{Sh}(\cbicat, J)$ under the equivalence of Proposition \ref{prop.sitoperslice}.	
\end{cor}
\begin{proof}
This follows from Theorem \ref{dp.groth} by taking $f$ equal to the arrow $!_{P}:P\to 1_{\textup{Sh}(\cbicat, J)}$ (notice that $\int (1_{\textup{Sh}(\cbicat, J)}:{\cal C}^{\textup{op}} \to \Set ) \simeq \cbicat$ and $\int !_{P}\simeq \pi_{P}$).
\end{proof}

\begin{rmk}
We can derive from Theorem \ref{dp.groth} a further description of $L_P^J$. Consider the unit arrow $\eta_{P}:P\rightarrow \sheafify P$. Then $C(\textstyle\int \eta_{P})_* \circ  R^J_P=R^J_{\sheafify P}\circ \prod_{\sheafify (\eta_{P})}$. But since $\sheafify(\eta_{P})$ is an isomorphism $\prod_{\sheafify(\eta_{P})}$ is an equivalence, so, up to isomorphism, $C(\textstyle\int \eta_{P})_* \circ R_P^J=R_{\sheafify P}^J$ and hence $L_P^J=L_{\sheafify P}^J \circ  C(\textstyle\int \eta_{P})_*$. Now, since $\sheafify P$ is a $J$-sheaf, $L_{\sheafify P}^J$ is the restriction of $L_{\sheafify P}$ (cf. Remark \ref{rmkrestrictiontosheaves}), while $C(\textstyle\int \eta_{P})_*$, as we know, is the restriction of $\ran_{\left(\int\eta_{P}\right)\op}$. Thus for $W:\left(\int P\right)\op\rightarrow\Set$ a $J_P$-sheaf and $X\in\cbicat$, 
$$L_P^J(W)(X)=(L_{\sheafify P}^J \circ C(\textstyle\int \eta_{P})_*)(W)(X)=\displaystyle\coprod_{x\in \sheafify P(X)} \lim_{\substack{\phi:Y\rightarrow X\\ p\in (\eta_{P})_Y\inv(\sheafify P(\phi)(x))} } W(Y,p).$$
\end{rmk}

The following corollary of Theorem \ref{dp.groth} provides an explicit description of the dependent product in the Grothendieck case.

\begin{cor}\label{corexplicitdesc}
	Let $({\cal C}, J)$ be an essentially small site and $f:P\to Q$ an arrow in $[{\cal C}^\textup{op}, \Set]$. Then  $$\textstyle\prod_{\sheafify(f)}=L^J_Q\circ C(\int f)_{\ast}\circ R^{J}_{P}\cong \sheafify_{Q}\circ L_{Q}\circ \ran_{(\int f)\op}\circ R_{P}\circ i_{P}\cong  \sheafify_{Q}\circ \textstyle\prod^{\textup{pr}}_{f} \circ i_{P},$$ where $\textstyle\prod^{\textup{pr}}_{f}:[{\cal C}^{\textup{op}}, \Set]/P \to [{\cal C}^{\textup{op}}, \Set]/Q$ is the dependent product along $f$ in the presheaf topos $[{\cal C}^{\textup{op}}, \Set]$.
	
	In particular, if $P$ and $Q$ are $J$-sheaves, then for any arrow $h:H\to P$ in $\textup{Sh}({\cal C}, J)$ and object $X$ in ${\cal C}$, we have
	\begin{equation*}\label{depprodgroth}
	\textstyle\prod_f[h](X)=\{(q,\underline{x})\ |\ q\in Q(X),\ \underline{x}\in A_{f}^{h}(X)  \}
	\end{equation*}
	(with the canonical projection to $Q$), where $A_{f}^{h}(X)$ is the set
	$$\{ \underline{x}\in\hspace{-1.7ex}\prod_{\substack{g:Y\rightarrow X\\ p\in f_Y\inv(Q(g)(q))}}\hspace{-1.7ex} h_Y\inv(p)\ |\ \forall \gamma\mbox{ s.t. }\cod(\gamma)=\dom(g),\ H(\gamma)(x_{g,p})=x_{g\circ \gamma, P(\gamma)(p)} \}$$
\end{cor}
	
\begin{proof}
The first statement of the corollary follows immediately from Theorem \ref{dp.groth}. 

The right Kan extension $\ran_{(\int f)\op}(R^J_P[h])$ can be computed (see \eg \cite[Example A4.1.4]{elephant}) as 
\begin{equation*}
\ran_{(\int f)\op}(R^J_P[h])(X,q)\hspace{-0.1ex}=\hspace{-0.1ex}\lim\left(\hspace{-0.1ex}\left((X,q)\hspace{-0.2ex}\downarrow\hspace{-0.2ex}\left(\textstyle\int\hspace{-0.2ex} f\right)\op\right)\hspace{-0.1ex}\rightarrow\hspace{-0.1ex} \left(\textstyle\int\hspace{-0.2ex} P\right)\op\hspace{-0.2ex}\xrightarrow{R^J_P[h]}\Set\right):
\end{equation*}
the nodes of the diagram are the arrows $g:(Y,f_Y(p))\rightarrow(X,q)$ in $\int Q$, so they are in fact indexed by pairs $(g:Y\rightarrow X, p\in P(Y))$ such that $f_Y(p)=Q(g)(q)$; the edges of the diagram are the arrows $\gamma:Y' \rightarrow Y$ in $\cal C$ such that $g'=g\circ \gamma$ and $P(\gamma)(p)=p'$ (notice that $\gamma$ is actually an arrow $(Y', f_{Y'}(p'))\rightarrow(Y, f_Y(p))$ by the latter conditions and the naturality of $f$). Recalling that $R^J_P([h])(Y,p)=h_Y\inv(p)$ if $P$ is a $J$-sheaf (cf. Remark  \ref{rmkrestrictiontosheaves}), the above limit is the set $A_{f}^{h}(X)$ in the statement of the corollary.
Finally, the coproduct of all these fibers yields the desired expression of the dependent product $\prod_f[h]$.	
\end{proof}	

Notice that we would have obtained exactly the same description for $\textstyle\prod_f$ as that provided by the Corollary by applying the construction of Section 1. Indeed, once recalled that $\powerset(Z)(X)$ can be interpreted as $\textup{Sub}(\yo(X)\times Z)$ in $[{\cal C}^{\textup{op}}, \Set]$ (for any $Z$), an explicit computation shows that 
\begin{align*}
\textstyle\prod_f[h](X)\hspace{-0.3ex}=\hspace{-0.3ex}\{&(q,M)\hspace{-0.2ex}\in\hspace{-0.2ex} Q(X)\hspace{-0.3ex}\times\hspace{-0.3ex} \textup{Sub}(\yo(X)\hspace{-0.3ex}\times\hspace{-0.3ex} H\hspace{-0.3ex}\times\hspace{-0.3ex} P)\hspace{0.2em} |\hspace{0.2em}  \forall g\hspace{-0.2ex}:\hspace{-0.2ex}Y\hspace{-0.3ex}\rightarrow\hspace{-0.3ex} X,\,\forall p\hspace{-0.2ex}\in\hspace{-0.2ex} P(Y)\\
&\hspace{0.3em}\mbox{ if}\hspace{0.2em} f(Y)(p)=Q(g)(q)\mbox{ then }\exists! x\in H(Y)\mbox{ s.t.}\hspace{0.2em} (g,x,p)\in M(Y);\\
&\hspace{2.2em} (g,x,p)\in M(Y)\mbox{ iff } h(Y)(x)=p\mbox{ and } f(Y)(p)=Q(g)(q)\}
\end{align*}
The parallelism with \ref{depprodgroth} is evident: indeed, to the tuple $\underline{x}$ we can associate the subobject $M\rightarrowtail \yo(X)\times H\times P$ of the triples $(g, x_{g,p}, p)$ and viceversa.

\begin{rmk}
	Corollary \ref{corexplicitdesc} shows that, if $P$ and $Q$ are $J$-sheaves, the dependent product $\prod_f$ is computed regardless of the topological \emph{datum} of $J$. This is due to the fact that, in this case, $R^{J}_{P}$ and $L^{J}_{Q}$ are restrictions of $R_P$ and $L_Q$ respectively (cf. Remark \ref{rmkrestrictiontosheaves}).
\end{rmk}

\textbf{Acknowledgements:} The first author gratefully acknowledges MIUR for the support in the form of a ``Rita Levi Montalcini'' position and Laurent Lafforgue for his careful reading of a preliminary version of this text.

\vspace{1.5cm}

\textsc{Olivia Caramello and Riccardo Zanfa} 

{\small \textsc{Dipartimento di Scienza e Alta Tecnologia, Universit\`a degli Studi dell'Insubria, via Valleggio 11, 22100 Como, Italy.}\\
	\emph{E-mail addresses:} \texttt{olivia.caramello@uninsubria.it}; \texttt{rzanfa@uninsubria.it.}}

\end{document}